\documentclass[12pt]{amsart}
 
\usepackage{graphicx} 
\usepackage{subfig}
\usepackage{amsmath}
\usepackage{amssymb}
\usepackage{mathrsfs}
\usepackage{diagbox}
\usepackage{amscd}
\usepackage{color}
\usepackage{soul}

\usepackage{enumerate}
\usepackage{multirow}
\usepackage{hyperref}
\usepackage{rotating}
\usepackage{tikz,mathdots}

\newcommand{\calC}{\mathcal C}
\newcommand{\calP}{\mathcal P}
\newcommand{\calS}{\mathcal S}
\newcommand{\calZ}{\mathcal Z}
\newcommand{\calW}{\mathcal W}
\newcommand{\R}{\mathbb R}
\newcommand{\N}{\mathbb N}

\newcommand{\Z}{\mathbb Z}

\newcommand{\be}{\begin{equation}}
\newcommand{\ee}{\end{equation}}

\newcommand{\citep}[1]{\cite{#1}}
\DeclareMathOperator{\vol}{vol}
\DeclareMathOperator{\std}{std}
\DeclareMathOperator{\cs}{cs}
\DeclareMathOperator{\des}{des}
\DeclareMathOperator{\pos}{pos}
\newcommand{\ind}[1]{\mathbf{1}_{#1}}

\newtheorem{thm}{Theorem}[section]

\newtheorem{lem}[thm]{Lemma}
\newtheorem{prop}[thm]{Proposition}
\theoremstyle{definition}
\newtheorem{defn}[thm]{Definition}
\newtheorem{rem}[thm]{Remark}

\newtheorem{ex}[thm]{Example}

\title[Extensions of partial cyclic orders and polytopes]{Extensions of partial cyclic orders and consecutive coordinate polytopes}

\author[A. Ayyer]{Arvind Ayyer}
\address{Arvind Ayyer,
Department of Mathematics, 
Indian Institute of Science,
Bangalore - 560012, India}
\email{arvind@iisc.ac.in}
\author[M. Josuat-Verg\`es]{Matthieu Josuat-Verg\`es}
\address{Matthieu Josuat-Verg\`es, Laboratoire d'Informatique Gaspard Monge, CNRS and Universit\'e Paris-Est Marne-la-Vall\'ee, France}
\email{matthieu.josuat-verges@u-pem.fr}
\author[S. Ramassamy]{Sanjay Ramassamy}
\address{Sanjay Ramassamy,
Unit\'e de Math\'ematiques Pures et Appliqu\'ees,
\'Eco\-le normale sup\'erieure de Lyon,
46 all\'ee d'Italie,
69364 Lyon Cedex 07, France}
\email{sanjay.ramassamy@ens.fr}

\date{\today}  

\begin{document} 

\begin{abstract}
We introduce several classes of polytopes contained in $[0,1]^n$ and cut out by inequalities involving sums of consecutive coordinates. We show that the normalized volumes of these polytopes enumerate circular extensions of certain partial cyclic orders. Among other things this gives a new point of view on a question popularized by Stanley. We also provide a combinatorial interpretation of the Ehrhart $h^*$-polynomials of some of these polytopes in terms of descents of total cyclic orders. The Euler numbers, the Eulerian numbers and the Narayana numbers appear as special cases.

\end{abstract}


\maketitle

\section{Introduction}
\label{sec:introduction}

Lattice polytopes, i.e. polytopes with vertices in $\Z^n$, have a volume which is an integer multiple of $1/n!$, which is the volume of the smallest simplex with vertices in $\Z^n$. An important question is to find a combinatorial interpretation of the integers arising as the normalized volume (the volume multiplied by factorial of the dimension) of some natural classes of lattice polytopes. The most celebrated instance is probably the Chan-Robbins-Yuen polytope~\cite{CRY00}, the normalized volume of which was conjectured by~\cite{CRY00} and shown by Zeilberger~\cite{Z99} to be equal to a product of Catalan numbers. This was later generalized to flow polytopes, see for example~\cite{CKM17} and the references therein. Another class of polytopes is that of the poset polytopes~\cite{S86}: to any poset one can associate two polytopes, the order polytope and the chain polytope of the poset, whose normalized volumes are equal to the number of linear extensions of the poset.

Refined enumeration results involve the Ehrhart $h^*$-polynomial of the polytope, which has the property that its coefficients are non-negative integers which sum to the normalized volume of the polytope~\cite{S80}. See~\cite{BR15} for some background about Ehrhart theory.

In this article, we associate natural polytopes to partial cyclic orders in the spirit of the chain polytopes construction~\cite{S86}. We define several classes of polytopes, obtained as subsets of $[0,1]^n$ and cut out by inequalities comparing the sum of some consecutive coordinates to the value $1$. Stanley asked for a formula of the normalized volumes of some of these polytopes in~\cite[Exercise~4.56(d)]{S12}. We show that the normalized volumes of these polytopes enumerate extensions of some partial cyclic orders to total cyclic orders (see below for some background on cyclic orders). We also find a combinatorial interpretation of the Ehrhart $h^*$-polynomials of some of these polytopes in terms of descents in the total cyclic orders. Remarkably enough, the Euler up/down numbers and the Eulerian numbers both arise, the former as the volumes of some polytopes and the latter as the coefficients of the $h^*$-polynomials of other polytopes. The Catalan and Narayana numbers also arise, as limiting values for the volumes and coefficients of the $h^*$-polynomials of a certain class of polytopes. Some of the polytopes we introduce belong to the class of Gorenstein polytopes (see e.g. \cite{BN08}).

A {\em cyclic order} $\pi$ on a set $X$ is a subset of triples in $X^3$ satisfying the following three conditions, respectively called cyclicity, asymmetry and transitivity:
\begin{enumerate}
  \item $\forall x, y, z \in X$, $(x,y,z) \in \pi \Rightarrow (y,z,x) \in \pi$;
  \item $\forall x ,y,z \in X$, $(x,y,z) \in \pi \Rightarrow (z,y,x) \not \in \pi$;
  \item $\forall x,y,z,u \in X$, $(x,y,z) \in \pi$ and $(x,z,u) \in \pi \Rightarrow (x,y,u) \in \pi$.
\end{enumerate}
A cyclic order $\pi$ is called {\em total} if for every triple of distinct elements $(x,y,z) \in X^3$, either $(x,y,z)\in \pi$ or $(z,y,x) \in \pi$. Otherwise, it is called {\em partial}.
Intuitively a total cyclic order $\pi$ on $X$ is a way of placing all the elements of $X$ on a circle such that a triple $(x,y,z)$ lies in $\pi$ whenever $y$ lies on the cyclic interval from $x$ to $z$ when turning around the circle in the clockwise direction. This provides a bijection between total cyclic orders on $X$ and cyclic permutations on $X$. See Figure~\ref{fig:cyclicorder7} for an example. This graphical representation is more intricate in the case of a partial cyclic order $\pi$, where there are usually multiple ``circles'' and each element may lie on several circles, as dictated by the triples belonging to $\pi$. Given a partial cyclic order $\pi'$ on $X$ and a total cyclic order $\pi$ on $X$, $\pi$ is called a \emph{circular extension} of $\pi'$ if $\pi' \subset \pi$. In other words, a circular extension of a partial cyclic order is a total cyclic order compatible with it.

In this article, we consider classes of total cyclic orders on $\{0,\ldots,n\}$ where we prescribe the relative position on the circle of certain consecutive integers. This amounts to looking at the set of all the circular extensions of a given partial cyclic order. Although the set of total cyclic orders on $\{0,\ldots,n\}$ is naturally in bijection with the set of permutations on $\{1,\ldots,n\}$, the conditions defining the subsets under consideration are expressed more naturally in terms of circular extensions. Not every partial cyclic order admits a circular extension, as was shown by Megiddo~\cite{M76}. The classes of partial cyclic orders considered in this article build upon those introduced in~\cite{R17}, which are the first classes for which positive enumerative results of circular extensions were obtained.

\begin{figure}[htbp]
\includegraphics[width=1.5in]{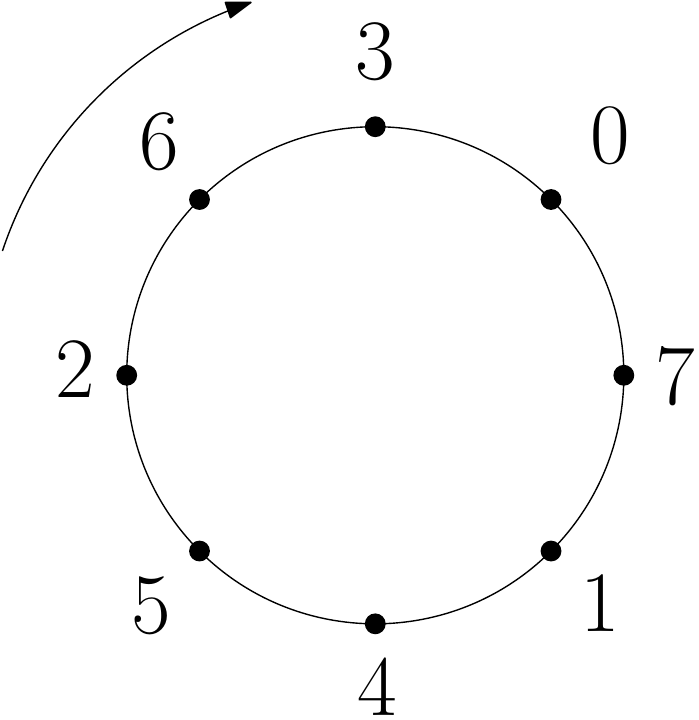}
\caption{An example of a graphical representation of a total cyclic order on $\{0,\ldots,7\}$. The arrow indicates the clockwise direction of rotation on the circle. This total cyclic order contains for example the triples $(0,4,2)$ and $(6,1,2)$ but not the triple $(7,5,4)$.}
\label{fig:cyclicorder7}
\end{figure}

\subsection*{Organization of the paper}

In Section~\ref{sec:mainresults}, we introduce several classes of polytopes and partial cyclic orders, then state the main results relating the volumes and Ehrhart $h^*$-polynomials of the former to the enumeration and refined enumeration of the circular extensions of the latter. In Section~\ref{sec:lattice} we prove that our polytopes are lattice polytopes. In Section~\ref{sec:transfermap} we introduce a transfer map which maps the original polytopes to sets whose volume can be computed easily, from which we deduce the statement about their volumes. Section~\ref{sec:Ehrhart} is mainly devoted to the interpretation of the coefficients of the Ehrhart $h^*$-polynomials of some of the polytopes. We also show in that section that some of these polynomials are palindromic. In Section~\ref{sec:Narayana} we show that a certain class of these polynomials stabilizes to the Narayana polynomials. Finally in Section~\ref{sec:boustrophedon} we explain how to use the multidimensional boustrophedon construction to compute the volumes of the polytopes.

\section{Main results}
\label{sec:mainresults}

\subsection{Volumes of polytopes}

For $n\geq1$, we denote by $[n]$ the set $\left\{0,1,\ldots,n\right\}$ and by $\calZ_n$ the set of total cyclic orders on $[n]$. If $m\geq3$ and $\pi\in\calZ_n$, the $m$-tuple $(x_1,\ldots,x_m)$ of distinct elements of $[n]$ is called a \emph{$\pi$-chain} if for $2\leq i \leq m-1$, we have $(x_1,x_i,x_{i+1})\in \pi$. In words, this means that if we place all the numbers on a circle in the cyclic order prescribed by $\pi$ and turn around the circle in the clockwise direction starting at $x_1$, we will first see $x_2$, then $x_3$, etc, before coming back to $x_1$. We extend this definition to the case $m=2$ by declaring that every pair $(x_1,x_2)\in [n]^2$ with $x_1\neq x_2$ forms a $\pi$-chain. For example, for the total cyclic order $\pi$ depicted on Figure~\ref{fig:cyclicorder7}, $(0,1,2,3)$ and $(1,5,6,3,7)$ are $\pi$-chains but $(1,2,3,4)$ is not a $\pi$-chain.

For $1\leq k \leq n$, define $\hat{A}_{k,n}$ to be the set of total cyclic orders $\pi\in\calZ_n$ such that for $0\leq i \leq n-k$, the $(k+1)$-tuple $(i,i+1,\ldots,i+k)$ forms a $\pi$-chain. This can easily be reformulated by saying that $\hat{A}_{k,n}$ is the set of all circular extensions of some given partial cyclic order.

For $1\leq k\leq n$, define the convex polytope $\hat{B}_{k,n}$ as the set of all $(x_1,\ldots,x_n)\in[0,1]^n$ such that
\[
\text{for } 0 \leq i \leq n-k, \text{ we have } x_{i+1} + \cdots + x_{i+k}\leq 1.
\]
These polytopes were introduced in~\cite[Exercise~4.56(c)]{S12} by Stanley. In the solution to this exercise, he gives some discrete difference equations for polynomials which can be used to compute the volumes of $\hat{B}_{k,n}$. He then asks for a formula for the normalized volumes of $\hat{B}_{k,n}$~\cite{stanley-2012}, \cite[Exercise 4.56(d)]{S12}. The polytopes $\hat{B}_{2,n}$ (resp. $\hat{B}_{3,n}$) seem to have been first considered in~\cite{SMN79} (resp. by Ki-rillov~\cite[Sequence A096402]{OEIS}). The polytopes $\hat{B}_{2,n}$ were also extensively studied by Diaconis and Wood~\cite{DW13}, arising as spaces of random doubly stochastic tridiagonal matrices.

Our first result relates the normalized volumes of $\hat{B}_{k,n}$ to the enumeration of $\hat{A}_{k,n}$.

\begin{thm}
\label{thm:Stanleyvolumes}
For $1\leq k\leq n$, the polytope $\hat{B}_{k,n}$ is a lattice polytope and we have
\begin{equation}
n! \vol(\hat{B}_{k,n}) = \# \hat{A}_{k,n}.
\end{equation}
\end{thm}

\begin{rem}
\label{rem:k23}
The cases $k=1$ and $k=2$ are already known. When $k=1$, $\hat{A}_{1,n}=\calZ_n$, which has cardinality $n!$ and $\hat{B}_{1,n}=[0,1]^n$, which has volume $1$. When $k=2$, it was shown in~\cite{R17} that $\# \hat{A}_{2,n}$ is equal to the $n$th \emph{Euler up/down number} $E_n$. On the other hand, it follows from~\cite[Exercise 4.56(c)]{S12} that $n! \vol(\hat{B}_{2,n}) = E_n$ (see also~\cite{SMN79}).
\end{rem}

Theorem~\ref{thm:Stanleyvolumes} admits a generalization where the lengths of the chains defining the partial cyclic order (resp. the number of coordinates appearing in each inequality defining the polytope) do not have to be all equal. For $n\geq1$, let $P_n$ be the set of all pairs $(i,j)\in [n]^2$ such that $i<j$. To every subset $I\subset P_n$, we associate the set $A_{I,n}$ of all the total cyclic orders $\pi \in\calZ_n$ such that for $(i,j)\in I$, $(i,i+1,\ldots,j)$ forms a chain in $\pi$. The set $A_{I,n}$ can be seen as the set of all the circular extensions of a given partial cyclic order. Furthermore, to every subset $I\subset P_n$, we associate the polytope $B_{I,n}$ defined as the set of all $(x_1,\ldots,x_{n})\in[0,1]^{n}$ such that
\[
\text{for } (i,j)\in I, \text{ we have } x_{i+1} + \cdots + x_{j-1} + x_j\leq 1.
\]
Then we have:

\begin{thm}
\label{thm:consecutivesums}
For $n\geq1$ and $I\subset P_{n}$, the polytope $B_{I,n}$ is a lattice polytope and we have
\begin{equation}
n! \vol(B_{I,n}) = \# A_{I,n}.
\end{equation}
\end{thm}

If $I=\left\{(i,i+k)\right\}_{0\leq i \leq n-k}$, we recover $A_{I,n}=\hat{A}_{k,n}$ and $B_{I,n}=\hat{B}_{k,n}$. Hence Theorem~\ref{thm:Stanleyvolumes} follows as a corollary of Theorem~\ref{thm:consecutivesums}.

\begin{rem}
If some pair $(i,j)\in I$ is nested inside another pair $(i',j')\in I$, then the condition on $\pi$-chains imposed by $(i,j)$ (resp. the inequality imposed by $(i,j)$) is redundant in the definition of $A_{I,n}$ (resp. $B_{I,n}$). Without loss of generality, we can thus restrict ourselves to considering sets $I$ with no nested pairs, which provides a minimal way of describing $A_{I,n}$ and $B_{I,n}$.
\end{rem}

The case $k=2$ of Theorem~\ref{thm:Stanleyvolumes} can be generalized in the following way. To every word $\underline{s}=(s_1,\ldots,s_n)\in\left\{+,-\right\}^n$ with $n\geq0$, following~\cite{R17}, one can associate the \emph{cyclic descent class} $\widetilde{A}_{\underline{s}}$, defined as the set of all $\pi\in\calZ_{n+1}$ such that for $1 \leq i\leq n$, we have $(i-1,i,i+1)\in \pi$ (resp. $(i+1,i,i-1)\in \pi$) if $s_i=+$ (resp. if $s_i=-$). The set $\widetilde{A}_{\underline{s}}$ can again be seen as the set of all the circular extensions of a given partial cyclic order. For example, if $s_i=+$ for $1 \leq i \leq n$, then $\widetilde{A}_{\underline{s}}=\hat{A}_{2,n+1}$. On the other hand, one can associate to every word $\underline{s}=(s_1,\ldots,s_n)\in\left\{+,-\right\}^n$ the polytope $\widetilde{B}_{\underline{s}}$ defined as the set of all $(x_1,\ldots,x_{n+1})\in[0,1]^{n+1}$ such that:
\begin{itemize}
 \item if $1 \leq i \leq n$ and $s_i=+$, then we have $x_i+x_{i+1}\leq 1$;
 \item if $1 \leq i \leq n$ and $s_i=-$, then we have $x_i+x_{i+1}\geq 1$.
\end{itemize}
For example, if $s_i=+$ for $1 \leq i \leq n$, then $\widetilde{B}_{\underline{s}}=\hat{B}_{2,n+1}$. We then have the following result.

\begin{thm}
\label{thm:cyclicdescents}
For $n\geq0$ and $\underline{s}=(s_1,\ldots,s_n)\in\left\{+,-\right\}^n$, the polytope $\widetilde{B}_{\underline{s}}$ is a lattice polytope and we have
\begin{equation}
(n+1)! \vol(\widetilde{B}_{\underline{s}}) = \# \widetilde{A}_{\underline{s}}.
\end{equation}
\end{thm}

\begin{rem}
The polytopes $\hat{B}_{2,n}$ arise as the chain polytopes of zigzag posets~\cite{S86}. For $k \leq n \leq 2k$, the polytopes $\hat{B}_{k,n}$ also arise as chain polytopes of some posets. This corresponds to the Catalan/Narayana range of parameters (see Theorem~\ref{thm:Narayana} and the paragraph following it). However we stress that in general, the polytopes $B_{I,n}$ and $\widetilde{B}_{\underline{s}}$ do not arise as chain polytopes of some posets. For example, one can show that $\hat{B}_{3,n}$ is not a chain polytope whenever $n\geq7$.
\end{rem}

\subsection{\texorpdfstring{Ehrhart $h^*$-polynomials}{Ehrhart h*-polynomials}}

One can refine Theorem~\ref{thm:consecutivesums} by considering the Ehrhart $h^*$-polynomials of the polytopes $B_{I,n}$, whose evaluations at $1$ give the volumes of the polytopes. The book~\cite{BR15} is a good reference for the basics of Ehrhart theory.

\begin{defn}
 If $P\subset \mathbb{R}^n$ is a lattice polytope, its Ehrhart function is defined for every integer $t\geq 0$ by
\[
E( P , t ) := \# (t\cdot P) \cap \mathbb{Z}^n
\]
where $t\cdot P$ is the dilation of $P$ by a factor $t$, i.e.~$t\cdot P = \{t\cdot   v \; | \;   v \in P \}$.
\end{defn}

This function may in fact also be defined if $P$ is an arbitrary bounded subset of $\mathbb{R}^n$ and this point of view will be useful later. When $P$ is a lattice polytope, the function $E(P,t)$ is actually a polynomial function of $t$. Hence it is called the \emph{Ehrhart polynomial} of $P$. 

\begin{defn}
 If $P\subset \mathbb{R}^n$ is a lattice polytope, we set
 \[
    E^*(P,z) := (1-z)^{n+1} \sum_{t=0}^\infty E(P,t) z^t.
 \]
The function $E^*(P,z)$ is a polynomial in $z$, called the \emph{Ehrhart $h^*$-polynomial} of $P$.
\end{defn}

By a result of Stanley~\cite{S80}, the coefficients of $E^*(P,z)$ are nonnegative integers whose sum equals the normalized volume of $P$. We provide a combinatorial interpretation of the coefficients of the $h^*$-polynomial of $B_{I,n}$ in terms of descents in the elements of $A_{I,n}$.

To every total cyclic order $\pi\in\calZ_n$ we associate the word $\underline{\pi}$ of length $n+1$ obtained by placing the elements of $\pi$ on a circle in the cyclic order imposed by $\pi$ and reading them in the clockwise direction, starting from $0$. For example, for the total cyclic order $\pi$ depicted on Figure~\ref{fig:cyclicorder7}, we have $\underline{\pi}=(0,7,1,4,5,2,6,3)$. We denote by $\calW_n$ the set of words of length $n+1$ with letters in $[n]$ that are all distinct and starting with $0$. Then $\pi \in\calZ_n \mapsto \underline{\pi} \in\calW_n$ is a bijection.

Given a word $\underline{w}=(w_0,\ldots,w_n)\in\calW_n$ and an integer $i$ such that $0 \leq i\leq n-1$, we say that $\underline{w}$ has a {\em descent} at position $i$ if $w_{i+1} < w_i$. We denote by $\des(\underline{w})$ the number of positions at which $\underline{w}$ has a descent. For example, the word $\underline{w}=(0,3,4,1,5,2)$ has two descents, at positions $2$ and $4$, and thus $\des(\underline{w})=2$. We have the following generalization of Theorem~\ref{thm:consecutivesums}.

\begin{thm}
\label{thm:Ehrhartconsecutivesums}
For $n\geq1$ and $I \subset P_n$, we have
\begin{equation}
\label{eq:Ehrhartdescent}
E^*(B_{I,n},z)=\sum_{\pi \in A_{I,n}} z^{\des(\underline{\pi})}.
\end{equation}
\end{thm}

In Figure~\ref{fig:numericaldata}, we display the polynomials $E^*(\hat{B}_{k,n},z)$ for some small values of $k$ and $n$. These polynomials have several remarkable features: they are palindromic, they contain the Eulerian polynomials and they stabilize in a certain limit to the Narayana polynomials.

Recall that a polynomial $R(z)=\sum_{h=0}^d a_h z^h$ of degree $d$ is called \emph{palindromic} if its sequence of coefficients is symmetric, i.e. for $0 \leq h \leq d$, we have $a_h=a_{d-h}$.

\begin{thm}
\label{thm:palindromicity}
For $1 \leq k\leq n$, the polynomial $E^*(\hat{B}_{k,n},z)$ is palindromic. 
\end{thm}

Note that in general, the polynomials $E^*(B_{I,n},z)$ are not palindromic.

\begin{rem}
In the case $k=1$, $\hat{B}_{1,n}$ is the unit hypercube $[0,1]^n$ and its $h^*$-polynomial is well-known to be the $n$th \emph{Eulerian polynomial}, whose coefficients enumerate the permutations of $\{1,\ldots,n\}$ refined by their number of descents (see e.g.~\cite{HJV16}). This is consistent with the fact that $\hat{A}_{1,n}$ is in bijection with the set of all permutations of $\{1,\ldots,n\}$, arising upon removing the initial $0$ from each word $\underline{\pi}$ for $\pi\in \hat{A}_{1,n}$.
\end{rem}

For $1\leq k\leq n$, define the \emph{Narayana numbers}~\cite[Sequence A001263]{OEIS})
\begin{equation}
\label{eq:narayana-def}
N(n,k):=\frac{1}{n} \binom n k \binom n{k-1},
\end{equation}
and the \emph{Narayana polynomials}
\begin{equation}
\label{eq:narayana-poly-def}
Q_n(z):=\sum_{k=1}^n N(n,k)z^{k-1}.
\end{equation}
The Narayana numbers are a well-known refinement of Catalan numbers, counting for example the number of Dyck paths with prescribed length and number of peaks (see e.g.~\cite[Example III.13]{FS09}). We have the following stabilization result of the Ehrhart $h^*$-polynomials of $\hat{B}_{k,n}$ to the Narayana polynomials.

\begin{thm}
\label{thm:Narayana}
For $1 \leq k \leq n \leq 2k$, we have
\[
E^*(\hat{B}_{k,n},z)=Q_{n-k+1}(z).
\]
\end{thm}

This result generalizes the fact that the normalized volume of $\hat{B}_{k,n}$ when $k \leq n \leq 2k$ is the $(n-k+1)$st Catalan number~\cite[Exercise 4.56(e)]{S12}.

\section{Lattice polytopes}
\label{sec:lattice}

In this section we show that the polytopes $B_{I,n}$ and $\widetilde{B}_{\underline{s}}$ are lattice polytopes by appealing to the theory of unimodular matrices. A rectangular matrix $M$ is said to be {\em totally unimodular} if every nonsingular square submatrix of $M$ is unimodular, i.e. has determinant $\pm 1$. By~\cite[Theorem 19.1]{schrijver-1986}, if $M$ is totally unimodular then for every integral vector $c$, the polyhedron defined by
\begin{equation}
\label{eq:polytopeequation}
\{x \: | \: x \geq 0, Mx \leq c \}
\end{equation}
is integral, i.e. it is equal to the convex hull of its integer points. In the case of polytopes, which are bounded polyhedra, the integrality property is equivalent to a polytope being a lattice polytope. Thus it suffices to realize the polytopes $B_{I,n}$ and $\widetilde{B}_{\underline{s}}$ as a set of inequalities in the form of~\eqref{eq:polytopeequation} involving a totally unimodular matrix $M$ to conclude that these polytopes are lattice polytopes. This is what we do now.

\begin{lem}
\label{lem:BIintegral}
For $n\geq 1$ and $I\subset P_n$, there exists a totally unimodular matrix $M_{I,n}$ and an integral vector $c_{I,n}$ such that
\[
B_{I,n}=\{x \: | \: x \geq 0, M_{I,n}x \leq c_{I,n} \}.
\]
\end{lem}

\begin{proof}
Fix $n\geq 1$ and $I\subset P_n$. Write
\[
I=\{ (i_1,j_1),\ldots,(i_m,j_m) \},
\]
where $m\geq1$ is the cardinality of $I$. Define
\[
M'_{I,n}:=\left(\ind{i_p < q \leq j_p} \right)_{\substack{1\leq p \leq m \\ 1 \leq q \leq n}}.
\]
In words, $M'_{I,n}$ is the $m\times n$ matrix such that for $1 \leq p \leq m$, the $p$th row of $M'_{I,n}$ contains a $1$ in positions located between $i_p+1$ and $j_p$ and $0$ elsewhere. Set $M_{I,n}$ to be the $(m+n)\times n$ matrix whose first $n$ rows consist of the identity and whose last $m$ rows consist of $M'_{I,n}$. Let $c_{I,n}$ be the vector in $\R^{m+n}$ with all coordinates equal to $1$.
Then
\[
B_{I,n}=\{x \: | \: x \geq 0, M_{I,n}x \leq c_{I,n} \}.
\]
The matrix $M_{I,n}$ has the property that it is a matrix with entries in $\{0,1\}$ where the $1$'s in each line are arranged consecutively.  Such matrices are called {\em interval matrices} and form a well-known class of totally unimodular matrices~\cite[Chapter 19, Example 7]{schrijver-1986}.
\end{proof}

\begin{lem}
\label{lem:Bsintegral}
For $n\geq 1$ and $\underline{s}\in\{ +,-\}^n$, there exists a totally unimodular matrix $M_{\underline{s}}$ and an integral vector $c_{\underline{s}}$ such that
\[
\widetilde{B}_{\underline{s}}=\{x \: | \: x \geq 0, M_{\underline{s}} x \leq c_{\underline{s}} \}.
\]
\end{lem}

\begin{proof}
Fix $n\geq 1$ and $s\in\{+,-\}^n$. Define $M'_{\underline{s}}$ to be the matrix of size $n \times (n+1)$ such that for $1 \leq i \leq n$, the entries in positions $(i,i)$ and $(i,i+1)$ of $M'_{\underline{s}}$ are equal to $1$ (resp. $-1$) if $s_i=+$ (resp. $s_i=-$), and all the other entries of $M'_{\underline{s}}$ are zero. Set $M_{\underline{s}}$ to be the $(2n+1) \times (n+1)$ matrix whose $n+1$ first rows consist of the identity matrix and whose last $n$ rows consist of $M'_{\underline{s}}$. Set $c_{\underline{s}}$ to be the vector in $\R^{2n+1}$ with the $(n+1+i)$th coordinate equal to $-1$ if $s_i=-$ for $1\leq i\leq n$ and all the other coordinates equal to $1$. Then
\[
\widetilde{B}_{\underline{s}}=\{x \: | \: x \geq 0, M_{\underline{s}} x \leq c_{\underline{s}} \}.
\]
The matrix $M_{\underline{s}}$ can be realized as an interval matrix (with entries in $\{0,1\}$) up to multiplying some rows by $-1$. Since interval matrices are totally unimodular, the matrix $M_{\underline{s}}$ is also totally unimodular.
\end{proof}

\section{The transfer map}
\label{sec:transfermap}

In this section we prove Theorem~\ref{thm:consecutivesums} and Theorem~\ref{thm:cyclicdescents} by introducing a transfer map $F_n$ from $[0,1]^n$ to itself, which is piecewise linear, (Lebesgue) measure-preserving and bijective outside of a set of measure $0$. We will show that the images under $F_n$ of the polytopes of types $B_{I,n}$ for $I\subset P_{n}$ and $\widetilde{B}_{\underline{s}}$ for $\underline{s}\in\left\{+,-\right\}^n$ are some sets whose normalized volumes are easily seen to enumerate the sets $A_{I,n}$ and $\widetilde{A}_{\underline{s}}$.

For $n\geq1$, we define the map
\[
F_n:(x_1,\ldots,x_n)\in[0,1]^n \rightarrow \left(\sum_{j=1}^i x_j \mod 1 \right)_{1 \leq i\leq n} \in [0,1)^n.
\]
In order to avoid confusion, we will denote the coordinates on the source (resp. target) of $F_n$ by $(x_1,\ldots,x_n)$ (resp. $(y_1,\ldots,y_n)$). 

\begin{lem}
\label{lem:measurepreserving}
The map $F_n$ induces a piecewise linear measure-preserving bijection from $[0,1)^n$ to itself.
\end{lem}

\begin{proof}
For $n\geq1$, define the map
\[
G_n:(y_1,\ldots,y_n)\in[0,1)^n \mapsto (x_1,\ldots,x_n) \in [0,1)^n,
\]
where $x_1:=y_1$ and for $2 \leq i \leq n$,
\[
x_i:=
\begin{cases}
y_i-y_{i-1} &\text{ if } y_i \geq y_{i-1}, \\
1+ y_i-y_{i-1} &\text{ if } y_i < y_{i-1}.
\end{cases}
\]
It is a straightforward check that $G_n$ is a left- and right-inverse of $F_n$ on $[0,1)^n$. Recall that $\N$ denotes the set of all positive integers. For $n\geq 1$, define the set of measure zero
\[
X'_n:=\Big\{(x_1,\ldots,x_n)\in[0,1)^n \; {\Big |} \; \exists j \in\{1,\ldots,n\} \; \text{so that} \; \sum_{i=1}^j x_i \in \N \Big\}.
\]
On each connected component of $[0,1)^n \setminus X'_n$, the map $F_n$ coincides with a translate of the map
\[
F'_n:(x_1,\ldots,x_n)\in[0,1)^n \rightarrow \left(\sum_{j=1}^i x_j \right)_{1 \leq i\leq n}.
\]
Since the matrix of $F'_n$ in the canonical basis is upper triangular with $1$ on the diagonal, $F'_n$ is a measure-preserving linear map and $F_n$ is also measure-preserving.
\end{proof}

\begin{rem}
A map very similar to $G_n$ was introduced in~\cite{S77} in order to show that the volumes of hypersimplices are given by Eulerian numbers.
\end{rem}

For $1\leq i \leq n$, given a word $\underline{w}=(w_0,\ldots, w_n)\in\calW_n$, we define $\pos_{\underline{w}}(i)$ (the position of $i$ in $\underline{w}$) to be the unique $j$ between $1$ and $n$ such that $w_j=i$. We associate to every element of $[0,1)^n$ two words, its standardization (following~\cite{HJV16}) and its cyclic standardization.

\begin{defn}
\label{defn:standardizations}
Let $n\geq1$ and let $y=(y_1,\ldots,y_n)\in[0,1)^n$.
\begin{enumerate}
 \item The \emph{standardization} of $y$, denoted by $\std(y)$, is defined to be the unique permutation $\sigma\in\calS_n$ such that for $1\leq i<j \leq n$, $\sigma(i) > \sigma(j)$ if and only if $y_i > y_j$. Using the one-line notation for permutations, the standardization of $y$ can also be seen as an $n$-letter word.
 \item The \emph{cyclic standardization} of $y$, denoted by $\cs(y)$, is defined to be the unique word $\underline{w}=(w_0,\ldots, w_n)\in\calW_n$ such that for $1\leq i<j \leq n$, $\pos_{\underline{w}}(i) > \pos_{\underline{w}}(j)$ if and only if $y_i > y_j$.
\end{enumerate}
\end{defn}

For example, if $y=(0.2,0.7,0.2,0.1,0.2)$, then $\std(y)=2 5 3 1 4$ and $\cs(y)=0 4 1 3 5 2$. The following result is an immediate consequence of Definition~\ref{defn:standardizations}.

\begin{lem}
\label{lem:standardization}
Let $n\geq1$ and let $y=(y_1,\ldots,y_n)\in[0,1)^n$. Write $\sigma=\std(y)\in\calS_n$. Then the word $\cs(y)$ is obtained by adding the letter $0$ in front of the word $(\sigma^{-1} (1), \ldots, \sigma^{-1} (n))$ representing the permutation $\sigma^{-1}$ in one-line notation.
\end{lem}

For $n\geq1$ and $\pi\in\calZ_n$, we define $S_\pi$ to be the set of all $y\in[0,1)^n$ whose cyclic standardization is $\underline{\pi}$:
\[
S_\pi:= \left\{y\in[0,1)^n \; | \; \cs(y)=\underline{\pi} \right\}.
\]
We have the following result about the sets $S_\pi$.
\begin{lem}
\label{lem:simplexvolume}
Let $n\geq1$ and let $\pi\neq \pi' \in \calZ_n$. The sets $S_\pi$ and $S_{\pi'}$ have disjoint interiors and
\[
\vol(S_\pi) = \frac{1}{n!}.
\]
\end{lem}

\begin{proof}
It is not hard to see that for $\pi\in\calZ_n$, the set $S_\pi$ is defined by $n+1$ inequalities. For example, if $\pi$ is such that $\underline{\pi}= 0 4 1 3 5 2$, then $S_\pi$ is defined by the inequalities
\[
0 \leq y_4 < y_1 \leq y_3 \leq y_5 < y_2 <1.
\]
The interior $\mathring{S}_\pi$ of $S_\pi$ is defined by making strict all the inequalities used to define $S_\pi$. It follows that if $\pi\neq \pi' \in \calZ_n$, then $\mathring{S}_\pi\cap\mathring{S}_{\pi'}=\emptyset$. Furthermore, by symmetry, all the $\mathring{S}_\pi$ where $\pi$ ranges over $\calZ_n$ have the same volume. Since
\[
[0,1)^n=\bigsqcup_{\pi\in\calZ_n} S_\pi,
\]
we deduce that $\vol(S_\pi)=\vol(\mathring{S}_\pi)=\frac{1}{n!}$ for $\pi\in\calZ_n$.
\end{proof}

For $n\geq1$ we define the sets
\[
X_n:=\big\{ (x_1,\ldots,x_n) \in (0,1)^n \mid \forall 1\leq i \leq j \leq n,\ 
x_i + \cdots + x_j \notin \Z \big\}
\]
and
\[
Y_n:=\big\{ (y_1,\ldots,y_n) \in (0,1)^n \mid \forall 1 \leq i < j \leq n,\ 
y_i \neq y_j \big\}.
\]
Both $X_n$ and $Y_n$ have full Lebesgue measure as subsets of $[0,1]^n$. Furthermore, $F_n$ maps $X_n$ to $Y_n$.

\begin{prop}
\label{prop:transfer}
For $n\geq1$ and $I\subset P_n$, we have
\begin{equation}
\label{eq:transferI}
F_n(B_{I,n}\cap X_n) = \bigsqcup_{\pi\in A_{I,n}} (S_\pi \cap Y_n).
\end{equation}
For $n\geq0$ and $\underline{s}\in\left\{+,-\right\}^n$, we have
\begin{equation}
\label{eq:transfers}
F_{n+1}(\widetilde{B}_{\underline{s}} \cap X_{n+1}) = \bigsqcup_{\pi\in \widetilde{A}_{\underline{s}}} (S_\pi \cap Y_{n+1}).
\end{equation}
\end{prop}

\begin{proof}
Let $\calC$ denote the circle obtained by quotienting out the interval $[0,1]$ by the relation $0 \sim 1$. The circle $\calC$ comes naturally equipped with the standard cyclic order. Fix $n\geq1$ and $(x_1,\ldots,x_n)\in X_n$. We set $(y_1,\ldots,y_n)=F(x_1,\ldots,x_n)\in Y_n$ and we let $\pi$ be the element of $\calZ_n$ such that $\cs(y)=\underline{\pi}$. Observe that each variable $x_i$ measures the gap between $y_{i-1}$ and $y_i$ when turning in the clockwise direction on $\calC$ (where by convention $y_0=0$). Thus, for $1 \leq i < j \leq n$, we have $x_{i+1}+\cdots+ x_j <1$ if and only if $(y_i,y_{i+1},\ldots,y_j)$ forms a $\calC$-chain. Furthermore, $(y_i,y_{i+1},\ldots,y_j)$ forms a $\calC$-chain if and only if $(i,i+1,\ldots,j)$ forms a $\pi$-chain. It follows immediately that for $I\subset P_n$, $x\in B_{I,n}$ if and only if $\pi\in A_{I,n}$. Equality~\eqref{eq:transferI} follows from the fact that $F_n$ is a bijection from $X_n$ to $Y_n$.

\begin{figure}[htbp]
\includegraphics[width=3in]{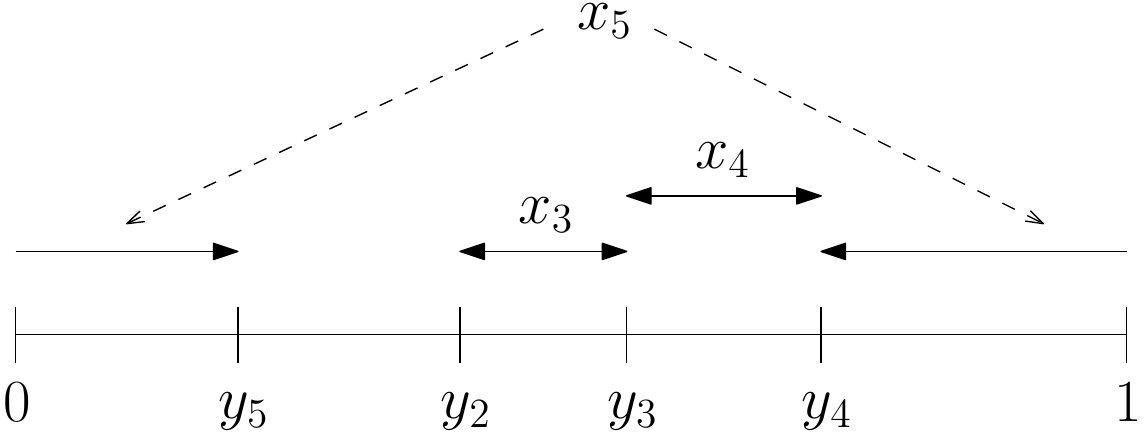}
\caption{The inequality $x_3+x_4+x_5<1$ is equivalent to the fact that $(y_2,y_3,y_4,y_5)$ forms a chain in $[0,1]/\sim$ equipped with its standard cyclic order.}
\label{fig:equivinequalitychain}
\end{figure}

Equality~\eqref{eq:transfers} follows similarly, by observing that the knowledge of the sign of $x_i+x_{i+1}-1$ is equivalent to the knowledge of the relative positions of $y_{i-1}$, $y_i$ and $y_{i+1}$ on $\calC$.
\end{proof}

Theorems~\ref{thm:consecutivesums} and~\ref{thm:cyclicdescents} follow immediately from combining Lemma~\ref{lem:measurepreserving}, Lemma~\ref{lem:simplexvolume} and Proposition~\ref{prop:transfer}. In order to compute the volumes of the polytopes $B_{I,n}$ and $\widetilde{B}_{\underline{s}}$ we were able to work up to sets of measure zero, discarding the complementary of the sets $X_n$ and $Y_n$. However we will need to take these sets of measure zero into account to compute the $h^*$-polynomial of $B_{I,n}$ in the next section.

\section{\texorpdfstring{The Ehrhart $h^*$-polynomial}{The Ehrhart h*-polynomial}}
\label{sec:Ehrhart}

In this section, we first prove Theorem~\ref{thm:Ehrhartconsecutivesums} about the combinatorial interpretation of the $h^*$-polynomial of $B_{I,n}$ in terms of descents of elements in $A_{I,n}$. Then we prove Theorem~\ref{thm:palindromicity} about the palindromicity of the $h^*$-polynomials of $\hat{B}_{k,n}$.

The first step of the proof of Theorem~\ref{thm:Ehrhartconsecutivesums} consists in relating our polytope $B_{I,n}$ with its ``half-open'' analog $B'_{I,n}$, which we now define.

\begin{defn}
To every subset $I\subset P_n$, we associate the polytope $B'_{I,n}$ defined as the set of all $(x_1,\ldots,x_{n})\in[0,1)^n$ such that for $(i,j)\in I$, we have $x_{i+1} + \cdots + x_{j-1} + x_j< 1$.
\end{defn}
Though $B'_{I,n}$ is not a polytope (it is obtained from the polytope $B_{I,n}$ by removing some faces), we can define its Ehrhart polynomial 
in the usual way by
\[
\forall t\in\N, E(B'_{I,n} , t ) :=  \# \big( t\cdot B'_{I,n} \cap \mathbb{Z}^n \big)
\]
and its $h^*$-polynomial by
\[
E^*(B'_{I,n} , z ) =  (1-z)^{n+1} \sum_{t=1}^{\infty} E(B'_{I,n} , t ) z^t.
\]
Note that there is no general result to guarantee that $E^*(B'_{I,n} , z )$ is a polynomial with nonnegative coefficients. However, in our setting we do have the following result.

\begin{lem}
If $t \in \N$, then $E(B'_{I,n} , t ) = E(B_{I,n}, t-1)$ and
\begin{equation}
\label{eq:halfopenEhrhart}
  E^*( B'_{I,n} , z ) = z\cdot E^*( B_{I,n} , z ).
\end{equation}
\end{lem}

\begin{proof}
The first equality follows from the fact that for $t \in\N$,
\[
 \big( t \cdot B'_{I,n} \big) \cap \mathbb{Z}^n = \big( (t-1) \cdot B_{I,n} \big) \cap \mathbb{Z}^n.
\]
To see why this is true, it suffices to notice that the condition $x_{i+1}+x_{i+2}+ \cdots + x_{j} < t$ is equivalent to $x_{i+1}+x_{i+2} + \cdots + x_{j} \leq t-1$ when $(x_i)_{1\leq i\leq n} \in \mathbb{Z}^n$. The second equality follows from the first one.
\end{proof}

It remains to show that $E^*( B'_{I,n} , z )$ is the descent generating function of $A_{I,n}$ (up to this factor $z$). The first step is to understand the behavior of the transfer map on some of the measure zero sets that were discarded in equality~\eqref{eq:transferI}.

\begin{prop}
\label{prop:transferhalfopen}
For $n\geq1$ and $I\subset P_n$, we have
\[
F_n(B'_{I,n}) = \bigsqcup_{\pi\in A_{I,n}} S_\pi.
\]
\end{prop}

The difference between the above proposition and Proposition~\ref{prop:transfer} is that now we need to take into account the cases when $x$ is not in $X_n$, wherein two coordinates of $F_n(x)$ become equal. The idea of the proof is to deal with such potential equalities by desingularizing, i.e. adding a small quantity to each coordinate of $x$ to make the coordinates of $y$ all distinct and to show that this desingularization does not change the cyclic standardization.

\begin{proof}
Let $x=(x_1,\ldots,x_n)\in[0,1)^n$ and write $y=F_n(x)$. For $\epsilon>0$, set $x_\epsilon=(x_1+\epsilon,\ldots,x_n+\epsilon)$. There exists $\epsilon_0>0$ such that for $0<\epsilon \leq \epsilon_0$, we have $x_\epsilon \in X_n$ and
\[
F_n(x_\epsilon)= (y_1 + \epsilon, y_2 + 2\epsilon, \dots, y_n + n \epsilon) \in Y_n.
\]
Moreover, there exists $0<\epsilon_1 \leq \epsilon_0$ such that for $0 < \epsilon \leq \epsilon_1$ and for $1\leq i<j \leq n$, we have $y_i > y_j$ if and only if $y_i + i \epsilon > y_j + j \epsilon$. Putting everything together, we obtain that $\cs(F_n(x))=\cs(F_n(x_\epsilon))$ whenever $\epsilon>0$ is small enough.

If $x\in B'_{I,n}$, then $x_\epsilon \in B_{I,n}\cap X_n$ for $\epsilon>0$ small enough. Hence, by Proposition~\ref{prop:transfer}, we have
\[
F_n(x_\epsilon)\in \bigsqcup_{\pi\in A_{I,n}} (S_\pi \cap Y_n) \subset  \bigsqcup_{\pi\in A_{I,n}} S_\pi,
\]
and by the previous paragraph, we obtain that $F_n(x) \in \sqcup_{\pi\in A_{I,n}} S_\pi$.

Conversely, assume $x\notin B'_{I,n}$. Then there exists $(i,j)\in I$ such that
\begin{equation}
\label{eq:ineqij}
x_{i+1}+\cdots+x_j >1.
\end{equation}
The same inequality involving the coordinates $i+1$ to $j$ as~\eqref{eq:ineqij} holds for every $x_\epsilon$ with $\epsilon>0$. Therefore, for $\epsilon>0$ small enough, we have $x_\epsilon \in X_n \setminus B_{I,n}$. Taking $\epsilon$ small enough so that $x_\epsilon \in X_n$ and $\cs(F_n(x))=\cs(F_n(x_\epsilon))$, we deduce from Proposition~\ref{prop:transfer} that $F_n(x_\epsilon)\in S_\pi \cap Y_n$ for some $\pi \notin A_{I,n}$, and hence $F_n(x) \notin \sqcup_{\pi\in A_{I,n}} S_\pi$.
\end{proof}

It follows from Proposition~\ref{prop:transferhalfopen} that
\[
E^*(F_n(B'_{I,n}),z)=\sum_{\pi\in A_{I,n}} E^*(S_\pi,z).
\]
Here $F_n(B'_{I,n})$ is not a lattice polytope, but its $h^*$-polynomial is defined in the same way as that of $B'_{I,n}$. Note also that the transfer map preserves integrality. More precisely, for $v \in [0,1)^n$ and an integer $t\geq1$, we have $ t\cdot v \in \mathbb{Z}^n$ if and only if $ t\cdot F_n(v) \in \mathbb{Z}^n$. Thus
\[
E^*(B'_{I,n},z)=E^*(F_n(B'_{I,n}),z).
\]
To conclude the proof of Theorem~\ref{thm:Ehrhartconsecutivesums}, it suffices to know the $h^*$-poly-nomial of $S_\pi$. This is the content of the following lemma. 
Recall that for $\underline{w} \in \calW_n$, $\des(\underline{w})$ counts the number of descents of $\underline{w}$.

\begin{lem}
\label{lem:HJV}
Let $n\geq 1$ and let $\pi\in\calZ_n$. Then
\begin{equation}
E^*(S_\pi , z ) = z^{\des(\underline{\pi})+1}.
\end{equation}
\end{lem}

\begin{proof}
By Lemma~\ref{lem:standardization}, all the elements in $S_\pi$ have the same standardization, which we denote by $\sigma\in\calS_n$. Moreover, if we define
\[
T_\sigma:=\{ y\in[0,1)^n \mid \std(y)=\sigma \},
\]
then Lemma~\ref{lem:standardization} implies that $S_\pi=T_\sigma$. It also implies that the word $\underline{\pi}$ has the same number of descents as the word $(\sigma^{-1}(1),\ldots,\sigma^{-1}(n))$, which is just the descent number $\des(\sigma^{-1})$ of the permutation $\sigma^{-1}$. 
We then appeal to~\cite[Lemma 4]{HJV16}, which says exactly that
\[
E^*(T_{\sigma} , z ) = z^{\des(\sigma^{-1})+1},
\]
thereby concluding the proof.
\end{proof}

We conclude this section with a proof of Theorem~\ref{thm:palindromicity} regarding the palindromicity of the polynomials $E^*(\hat{B}_{k,n})$. For $v, w \in \mathbb{R}^n$, let $\langle v, w \rangle$ denote the standard inner product.

\begin{defn}
If a lattice polytope $P$ contains the origin in its interior, the \emph{dual polytope} is defined to be
\[
P^* :=  \{ v \in \mathbb{R}^n \; | \; \langle v, w \rangle \geq -1 \text{ for all } w\in P \}.
\]
We say that $P$ is {\em reflexive} if $P^*$ is also a lattice polytope.
\end{defn}

Hibi's palindromic theorem~\cite{H92} states that when a lattice polytope $P$ contains an integral interior point $v$, the polynomial $E^*(P,z)$ has palindromic coefficients if and only if $P-v$ is reflexive. But in the present case the polytopes $\hat{B}_{k,n}$ do not contain any integral interior point, as their vertices are vectors only containing $0$'s and $1$'s. However, we can resort to the more general theory of Gorenstein polytopes (see e.g.~\cite{BN08}). A lattice polytope $P$ is said to be {\em Gorenstein} if there exists an integer $t\geq1$ and an integer vector $v$, such that $t\cdot P - v$ is reflexive. It is known that a lattice polytope $P$ is Gorenstein if and only if $E^*(P,z)$ is palindromic~\cite{BN08}.

\begin{proof}[Proof of Theorem~\ref{thm:palindromicity}]
Fix $1\leq k\leq n$ and denote by $P_{k,n}$ the polytope $(k+1)\cdot \hat{B}_{k,n} - (1,\dots,1)$. The elements $v\in P_{k,n}$ are characterized by the inequalities
\begin{align}
   v_i &\geq -1 \quad \text{ for } 1\leq i \leq n, \label{eq:dilate1} \\
   v_{i+1}+\dots+v_{i+k} &\leq 1  \quad \text{ for } 0\leq i \leq n-k. \label{eq:dilate2}
\end{align}
Clearly $P_{k,n}$ is a lattice polytope containing the origin as an interior point. Furthermore, since there are $1$'s and $-1$'s on the right-hand sides of inequalities~\eqref{eq:dilate1} and~\eqref{eq:dilate2}, the coefficients on the left-hand sides of~\eqref{eq:dilate1} and~\eqref{eq:dilate2} are the coordinates of the vertices of $P_{k,n}^*$. Hence $P_{k,n}^*$ is a lattice polytope. So $P_{k,n}$ is reflexive and $\hat{B}_{k,n}$ is Gorenstein, thus $E^*(\hat{B}_{k,n})$ is palindromic.
\end{proof}

\begin{rem}
The proof of Theorem~\ref{thm:palindromicity} we provide is based on purely geometric considerations involving $\hat{B}_{k,n}$. It would be interesting to understand this palindromicity result on the combinatorial level, by finding an involution on $\hat{A}_{k,n}$ sending an element $\pi$ to an element $\pi'$ such that $\des(\pi)+\des(\pi')=\deg E^*(\hat{B}_{k,n})$.
\end{rem}

\begin{figure}[h!tp] \centering  \scriptsize
\begin{sideways}
\begin{tabular}{|c|c|c|c|c|c|c|}
\hline
   \diagbox{$k$}{$n-k$}  & 0    & 1            & 2                 & 3                     & 4                       & 5                               \\ \hline
  1                   & $1$   & $z+1$        & $z^2+4z+1$        & $z^3+11z^2+11z+1$     & $z^4+26z^3+66z^2+26z+1$ & $z^5+57z^4+302z^3+302z^2+57z+1$  \\ \hline
  2                   & $1$   & $z+1$        & $z^2+3z+1$        & $z^3 + 7z^2 + 7z + 1$ & $z^4+14z^3+31z^2+14z+1$ & $z^5+26z^4+109z^3+109z^2+26z+1$   \\ \hline
  3                   & $1$   & $z+1$        & $z^2+3z+1$        & $z^3 + 6z^2 + 6z + 1$ & $z^4+14z^3+31z^2+14z+1$ & $z^5+19z^4+71z^3+71z^2+19z+1$      \\ \hline
  4                   & $1$   & $z+1$        & $z^2+3z+1$        & $z^3 + 6z^2 + 6z + 1$ & $z^4+10z^3+20z^2+10z+1$ & $z^5+16z^4+56z^3+56z^2+16z+1$       \\ \hline  
  5                   & $1$   & $z+1$        & $z^2+3z+1$        & $z^3 + 6z^2 + 6z + 1$ & $z^4+10z^3+20z^2+10z+1$ & $z^5+15z^4+50z^3+50z^2+15z+1$        \\ \hline
  6                   & $1$   & $z+1$        & $z^2+3z+1$        & $z^3 + 6z^2 + 6z + 1$ & $z^4+10z^3+20z^2+10z+1$ & $z^5+15z^4+50z^3+50z^2+15z+1$         \\ \hline
  7                   & $1$   & $z+1$        & $z^2+3z+1$        & $z^3 + 6z^2 + 6z + 1$ & $z^4+10z^3+20z^2+10z+1$ & $z^5+15z^4+50z^3+50z^2+15z+1$        \\ \hline
\end{tabular}
\end{sideways}
\hspace{1cm}
\begin{sideways}
\begin{tabular}{|c|c|c|}
\hline
  \diagbox{$k$}{$n-k$}  & 6                                             & 7   \\ \hline
  1                   & $z^6+120z^5+1191z^4+2416z^3+1191z^2+120z+1$   & $z^7+247z^6+4293z^5+15619z^4+15619z^3+4293z^2+247z+1$    \\ \hline
  2                   & $z^6+46z^5+334z^4+623z^3+334z^2+46z+1$        & $z^7+79z^6+937z^5+2951z^4+2951z^3+937z^2+79z+1$      \\ \hline
  3                   & $z^6+31z^5+191z^4+340z^3+191z^2+31z+1$        & $z^7+49z^6+472z^5+1365z^4+1365z^3+472z^2+49z+1$   \\ \hline
  4                   & $z^6+25z^5+140z^4+242z^3+140z^2+25z+1$        & $z^7+38z^6+322z^5+881z^4+881z^3+322z^2+38z+1$        \\ \hline  
  5                   & $z^6+22z^5+115z^4+195z^3+115z^2+22z+1$        & $z^7+32z^6+249z^5+656z^4+656z^3+249z^2+32z+1$         \\ \hline
  6                   & $z^6+21z^5+105z^4+175z^3+105z^2+21z+1$        & $z^7+29z^6+211z^5+540z^4+540z^3+211z^2+29z+1$        \\ \hline
  7                   & $z^6+21z^5+105z^4+175z^3+105z^2+21z+1$        & $z^7+28z^6+196z^5+490z^4+490z^3+196z^2+28z+1$          \\ \hline
  8                   & $z^6+21z^5+105z^4+175z^3+105z^2+21z+1$        & $z^7+28z^6+196z^5+490z^4+490z^3+196z^2+28z+1$          \\ \hline
  9                   & $z^6+21z^5+105z^4+175z^3+105z^2+21z+1$        & $z^7+28z^6+196z^5+490z^4+490z^3+196z^2+28z+1$          \\ \hline
\end{tabular}
\end{sideways}
\caption{Some values of the $h^*$-polynomial of $\hat{B}_{k,n}$. This illustrates in particular the stabilization property: the sequence of polynomials in a particular column is stationary.}
\label{fig:numericaldata}
\end{figure}

\section{Stabilization to Narayana polynomials}
\label{sec:Narayana}

In this section, we first prove Theorem~\ref{thm:Narayana} about the stabilization of the Ehrhart $h^*$-polynomials of $\hat{B}_{k,n}$ to the Narayana polynomials, as illustrated in Figure~\ref{fig:numericaldata}. This is done combinatorially, using the connection with $\hat{A}_{k,n}$. Then we provide some geometric insight as to why the $h^*$-polynomials of $\hat{B}_{k,n}$ stabilize.

\begin{proof}[Proof of Theorem~\ref{thm:Narayana}]
We will first prove the result for $\hat{A}_{n,2n}$ via a bijective correspondence with nondecreasing parking functions. An $(n+1)$-tuple of nonnegative integers $(p_0,\ldots,p_n)$ is called a \emph{nondecreasing parking function} if the following two conditions hold:
\begin{enumerate}
 \item for $0 \leq i\leq n-1$, we have $p_i \leq p_{i+1}$;
 \item for $0 \leq i \leq n$, we have $0 \leq p_i \leq i$.
\end{enumerate}
We denote by $\calP_n$ the set of all $(n+1)$-tuples that are nondecreasing parking functions. It is well-known that the cardinality of $\calP_n$ is the $(n+1)$st Catalan number; see \cite[Exercise 6.19(s)]{stanley-ec2}, for example. A nondecreasing parking function $(p_0,\ldots,p_n)$ is said to have an \emph{ascent} at position $0\leq i\leq n-1$ if $p_i < p_{i+1}$. It follows from~\cite[Corollary A.3]{schumacher-2009} that the number of nondecreasing parking functions in $\calP_n$ with $k$ ascents is the Narayana number $N(n+1,k+1)$. To complete the proof in the case of $\hat{A}_{n,2n}$, we will define a bijection $H_n$ from $\hat{A}_{n,2n}$ to $\calP_n$ such that the number of descents of $\pi \in \hat{A}_{n,2n}$ equals the number of ascents of $H_n(\pi)$.

We term the numbers in $\{0,\ldots,n\}$ as ``small numbers''. Given $\pi \in \hat{A}_{n,2n}$, we set $H_n(\pi)$ to be the $(n+1)$-tuple $(p_0,\ldots,p_n)$ defined as follows. For $0 \leq i \leq n$, $p_i$ is defined to be the first small number to the right of $n+i$ in the word $\underline{\pi}$ if there exists a small number to the right of $n+i$. Otherwise $p_i$ is defined to be $0$. In other words, visualizing $\pi$ as the placement of the numbers from $0$ to $2n$ on a circle, $p_i$ is the next small number after $n+i$, turning in the clockwise direction. For example, if $n=3$ and $\underline{\pi}=(0,4,5,1,2,6,3)$, then $H_n(\pi)=(0,1,1,3)$.

For $\pi\in \hat{A}_{n,2n}$, we first show that $H_n(\pi)$ is in $\calP_n$. To begin with, note that the entries $0,1,\dots,n$ appear in that order in $\underline{\pi}$, so $n$ is the rightmost number in $\{1,\ldots,n\}$ to appear. If $0 \leq i \leq n$, since $(i,i+1,\ldots,n+i)$ forms a chain in $\pi$, then $n+i$ must lie either to the right of $n$ or to the left of $i$ in $\underline{\pi}$, hence $p_i\leq i$. Let $0 \leq i \leq n-1$. If $n+i$ lies to the right of $n$ in $\underline{\pi}$, then $p_i=0$ and $p_i \leq p_{i+1}$. Otherwise, $n+i$ lies to the left of $i$ and from the fact that $(i+1,\ldots,n+i,n+i+1)$ forms a chain in $\pi$, we deduce that $n+i+1$ lies between $n+i$ and $i+1$. Thus $p_i \leq p_{i+1}$ again. This concludes the proof that $H_n(\pi)$ is in $\calP_n$.

Next, given $\underline{p}=(p_0,\ldots,p_n) \in \calP_n$, we define a total cyclic order $H'_n(\underline{p})\in \calZ_{2n}$ as follows. First we place all the small numbers on the circle in such a way that $(0,1,\ldots,n)$ form a chain. Then we place each number $n+i$ for $1 \leq i\leq n$ in the cyclic interval from $p_i-1$ to $p_i$ if $p_i\geq 1$ and in the cyclic interval from $n$ to $0$ if $p_i=0$. This determines for $1 \leq i \leq n$ the position of each number $n+i$ with respect to the small numbers. Since the sequence $(p_0,\ldots,p_n)$ is weakly increasing, it is possible to arrange the numbers $n+i$ for $1 \leq i\leq n$ in such a way that $(n,n+1,\ldots,2n)$ forms a chain. This determines uniquely the position of all the numbers on the circle and yields (by definition) $H'_n(\underline{p})$.

Now we need to check that $\pi:=H'_n(\underline{p}) \in \hat{A}_{n,2n}$. By construction, we already have that $(i,i+1,\ldots,n+i)$ are $\pi$-chains when $i=0$ and $i=n$. Fix $1 \leq i\leq n-1$. Then $(i,i+1,\ldots,n)$ (resp. $(n,n+1,\ldots,n+i)$) forms a $\pi$-chain, as a subchain of $(0,1,\ldots,n)$ (resp. $(n,n+1,\ldots,2n)$). In the case when $p_i=0$ then for $1 \leq j \leq i$ we also have $p_j=0$, so all the numbers $n+1, n+2,\ldots n+i$ are to the right of $n$ in $\underline{\pi}$ and in this order. Hence $(i,i+1,\ldots,n+i)$ is a $\pi$-chain in this case. In the case when $p_i>0$, then $ p_0 \leq p_1 \leq \cdots \leq p_i \leq i$ and all the numbers $n+1, n+2,\ldots n+i$ are to the left of $i$ in $\underline{\pi}$ and in this order. Hence $(i,i+1,\ldots,n+i)$ is a $\pi$-chain in this case too. This concludes the proof that $H'_n(\underline{p})$ belongs to $\hat{A}_{n,2n}$. Clearly, $H_n$ and $H'_n$ are inverses to each other, so $H_n$ is a bijection from $\hat{A}_{n,2n}$ to $\calP_n$.

Let $\pi \in \hat{A}_{n,2n}$ and write $\underline{p}=H_n(\pi)$. Then for $0\leq i\leq n-1$, the two numbers $n+i$ and $n+i+1$ are consecutive in $\underline{\pi}$ if and only if $p_i=p_{i+1}$. The fact that descents in $\underline{\pi}$ are in one-to-one correspondence with ascents in $H_n(\pi)$ follows from the observation that a descent in $\underline{\pi}$ is always from a number larger than $n$ to a small number. This concludes the proof of Theorem~\ref{thm:Narayana} for $\hat{A}_{n,2n}$.

As for $\hat{A}_{k,n}$ where $k$ is bigger that $n/2$, the cycle conditions ensure that $n-k, \dots,k$ are consecutive elements in the list (the statement is nontrivial if $n-k<k$, i.e. $n < 2k$). For $\pi \in \hat{A}_{k,n}$, we obtain an element $\pi' \in \hat{A}_{n-k,2n-2k}$ by replacing these consecutive entries by the number $n-k$ and by replacing each entry $i$ with $k+1 \leq i \leq n$ by the number $i+n-2k$. It is not difficult to see that this is a bijection which preserves the number of descents. Finally we note that $\pi' \in \hat{A}_{m,2m}$, where $m = n-k$, and we appeal to the first part of the proof.
\end{proof}

We point out that it is possible to prove Theorem~\ref{thm:Narayana} differently, by noting that, although the general polytopes $\hat{B}_{k,n}$ do not arise as chain polytopes~\cite{S86}, for $k \leq n \leq 2k$, the polytopes $\hat{B}_{k,n}$ do arise as chain polytopes for posets associated with skew Young diagrams. Associating a Dyck path to every linear extension of such a poset yields another proof of Theorem~\ref{thm:Narayana}.

It is worth observing that one can also see the stabilization property of the $h^*$-polynomials geometrically.

\begin{lem}
If $k > \frac {n-1}2$, we have $E^*(\hat{B}_{k,n},z) = E^*(\hat{B}_{k+1,n+1},z)$.
\end{lem}

\begin{proof}
First note that $ k > \frac {n-1}2 $ is equivalent to $k\geq n-k$, and so there exists an $\ell$ such that $k \geq \ell \geq n-k$. Consider the map $\alpha:\mathbb{Z}^{n+1} \to \mathbb{Z} \times \mathbb{Z}^{n}$ defined by 
\[
  \alpha(v_1,\dots,v_{n+1}) = ( v_\ell , (v_1,\dots,\widehat{v_\ell},\dots, v_{n+1}) ),
\]
where $\widehat{v_\ell}$ means that $v_\ell$ is omitted in the sequence. It is clearly bijective. We claim that for any integer $t\geq0$, we have
\begin{equation}
\label{decomp_tB}
\alpha\big( (t\cdot \hat{B}_{k+1,n+1}) \cap \mathbb{Z}^{n+1} \big) = \bigsqcup_{u=0}^t \{t-u\}\times \big( (u\cdot \hat{B}_{k,n}) \cap \mathbb{Z}^n \big).
\end{equation}
To see this, let $ v = (v_i)_{1\leq i \leq n} \in \mathbb{Z}^{n+1}$. By definition of $\hat{B}_{k+1,n+1}$, we have $v\in t\cdot \hat{B}_{k+1,n+1}$ if and only if $v_i\geq 0$ and
\begin{equation}
\label{defining_eq_tB}
v_{i+1} + \dots + v_{i+k+1} \leq t, \quad \text{ for all } 0\leq i\leq n-k.
\end{equation}
Then, note that $k\geq \ell \geq n-k$ ensures that $v_\ell$ appears in all the sums in~\eqref{defining_eq_tB}. So the condition~\eqref{defining_eq_tB} is equivalent to
\begin{equation}
\label{defining_eq_tB2}
v_{i+1} + \dots + \widehat{v_\ell} + \dots + v_{i+k+1} \leq t-v_\ell, \quad \text{ for all } 0\leq i\leq n-k.
\end{equation}
These equations precisely say that $ (v_1,\dots,\widehat{v_\ell},\dots, v_{n+1}) \in (t-v_\ell) \cdot \hat{B}_{k,n}$, (knowing that $v_i\geq 0$). Thus we get~\eqref{decomp_tB}. Therefore, we have
\[
   E( \hat{B}_{k+1,n+1} , t ) = \sum_{u=0}^t E( \hat{B}_{k,n} , u ). 
\]
By summing, we get
\[
  \sum_{t=0}^\infty E( \hat{B}_{k+1,n+1} , t ) z^t =  \sum_{t=0}^\infty
  \sum_{u=0}^t E( \hat{B}_{k,n} , u ) z^t = \sum_{u = 0}^\infty E( \hat{B}_{k,n} , u ) \frac{z^{u}}{1-z}.
\]
After multiplying by $(1-z)^{n+1}$, we get $E^*(\hat{B}_{k,n},z) = E^*(\hat{B}_{k+1,n+1},z)$.
\end{proof}

Note that in Equation~\eqref{decomp_tB}, $u$ ranges among the integers between $0$ and $t$.  But the argument in the proof also shows that
\[
  \alpha( \hat{B}_{k+1,n+1} )  =  \bigsqcup_{0\leq x \leq 1} \{1-x\} \times \big( x\cdot \hat{B}_{k,n} \big),
\]
where $x$ runs through the real numbers between $0$ and $1$. This precisely says that $\hat{B}_{k+1,n+1}$ is a cone over $\hat{B}_{k,n}$.

\section{\texorpdfstring{Enumerating $\hat{A}_{k,n}$}{Enumerating A kn }}
\label{sec:boustrophedon}

In this section we show how to use the multidimensional boustrophedon construction introduced in~\cite{R17} to compute the cardinalities of $\hat{A}_{k,n}$, which by Theorem~\ref{thm:Stanleyvolumes} are equal to the normalized volumes of $\hat{B}_{k,n}$.

For every total cyclic order $\pi\in\calZ_n$ and for every two elements $i\neq j$ of $[n]$, define the length of the arc from $i$ to $j$ in $\pi$ to be
\begin{equation}
L_\pi(i,j):=1 + \#\left\{h\in [n] | (i,h,j)\in \pi \right\}.
\end{equation}

\begin{ex}
Take $n=6$ and take the cyclic order $\pi$ associated with the cyclic permutation $0 3 5 1 6 2 4$. Then
\begin{align*}
L_\pi(3,5) &= 1;\\
L_\pi(3,2) &= 4;\\
L_\pi(2,3) &=3.
\end{align*}
\end{ex}

For $d\geq1$ and $N\geq d+1$, define the simplex of dimension $d$ and order $N$ to be
\begin{equation}
T_N^d:=\left\{(i_1,\ldots,i_{d+1})\in\N^{d+1} \; | \; i_1+\cdots + i_{d+1}=N \right\}.
\end{equation}
When $d=1$, $T_N^1$ is a row of $N-1$ elements. When $d=2$ (resp. $d=3$), $T_N^d$ is a triangle (resp. tetrahedron) of side length $N-2$ (resp. $N-3$). In general, $T_N^d$ is a $d$-dimensional simplex of side length $N-d$.

For $2 \leq k \leq n$ and $\underline{i}=(i_1,\ldots,i_k)\in T_{n+1}^{k-1}$, we define $\check{A}_{\underline{i}}$ to be the set of all $\pi\in \hat{A}_{k,n}$ such that the following conditions hold:
\begin{itemize}
 \item for $1 \leq j \leq k-1$, we have $L_\pi(n+j-k,n+1+j-k)=i_j$;
 \item $L_\pi(n,n+1-k)=i_k$.
\end{itemize}
It is not hard to see that
\begin{equation}
\hat{A}_{k,n}=\bigsqcup_{\underline{i}\in T_{n+1}^{k-1}} \check{A}_{\underline{i}}.
\end{equation}

Define $a_{\underline{i}}:=\# \check{A}_{\underline{i}}$. We will provide linear recurrence relations for the $(a_{\underline{i}})_{\underline{i}\in T_{n+1}^{k-1}}$, which are arrays of numbers indexed by some $T_N^d$. We first need to define linear operators $\Psi$ and $\Omega$ which transform one array of numbers indexed by some $T_N^d$ into another array of numbers, indexed by $T_{N+1}^d$ in the case of $\Psi$ and by $T_N^d$ in the case of $\Omega$.

We define the map $\tau$ which to an element $\underline{i}=(i_1,\ldots,i_{d+1})\in T_{N+1}^d$ associates the subset of all the $\underline{i'}=(i'_1,\ldots,i'_{d+1})\in T_N^d$ such that the following conditions hold:
\begin{itemize}
 \item $1 \leq i'_1 \leq i_1-1$;
 \item $i'_{d+1}=i_{d+1}+i_1-i'_1-1$;
 \item $i'_j=i_j$ for $2 \leq j \leq d$.
\end{itemize}

\begin{ex}
Case $d=2$, $N=6$. Then
\begin{align}
\tau(1,2,4)&=\emptyset \\
\tau(4,2,1) &= \left\{(1,2,3),(2,2,2),(3,2,1) \right\}.
\end{align}
\end{ex}

Define the map $\Psi$ which sends an array of numbers $(b_{\underline{i}})_{\underline{i}\in T_N^d}$ to the array of numbers $(c_{\underline{i}})_{\underline{i}\in T_{N+1}^d}$, where for $\underline{i}\in T_{N+1}^d$, we have
\begin{equation}
c_{\underline{i}}:=\sum_{\underline{i'}\in\tau(\underline{i})} b_{\underline{i'}}.
\end{equation}

\begin{figure}[htbp]
\includegraphics[height=1.5in]{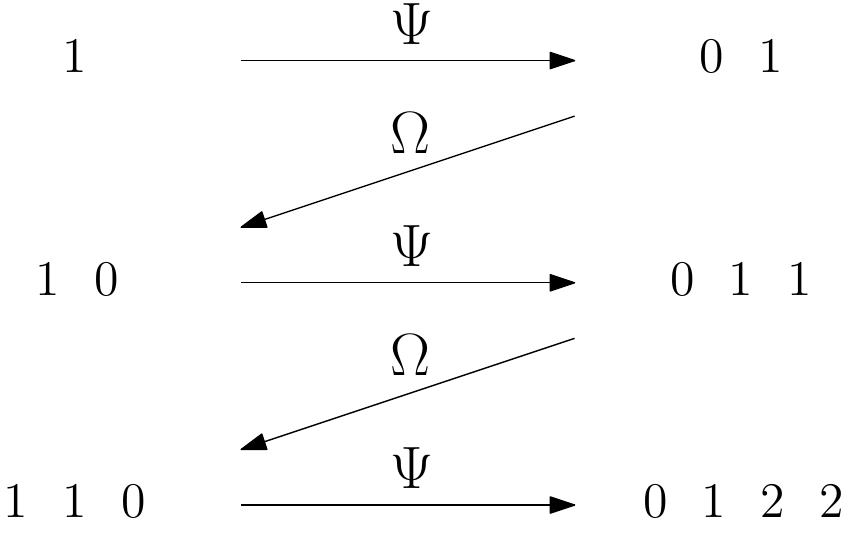}
\caption{The first four lines of the classical boustrophedon, obtained as the case $k=2$ of the multidimensional boustrophedon. Here the operator $\Psi$ corresponds to taking partial sums (starting from the empty partial sum) and the operator $\Omega$ returns the mirror image.}
\label{fig:boustrophedondim1}
\end{figure}

\begin{figure}[htbp]
\includegraphics[height=5in]{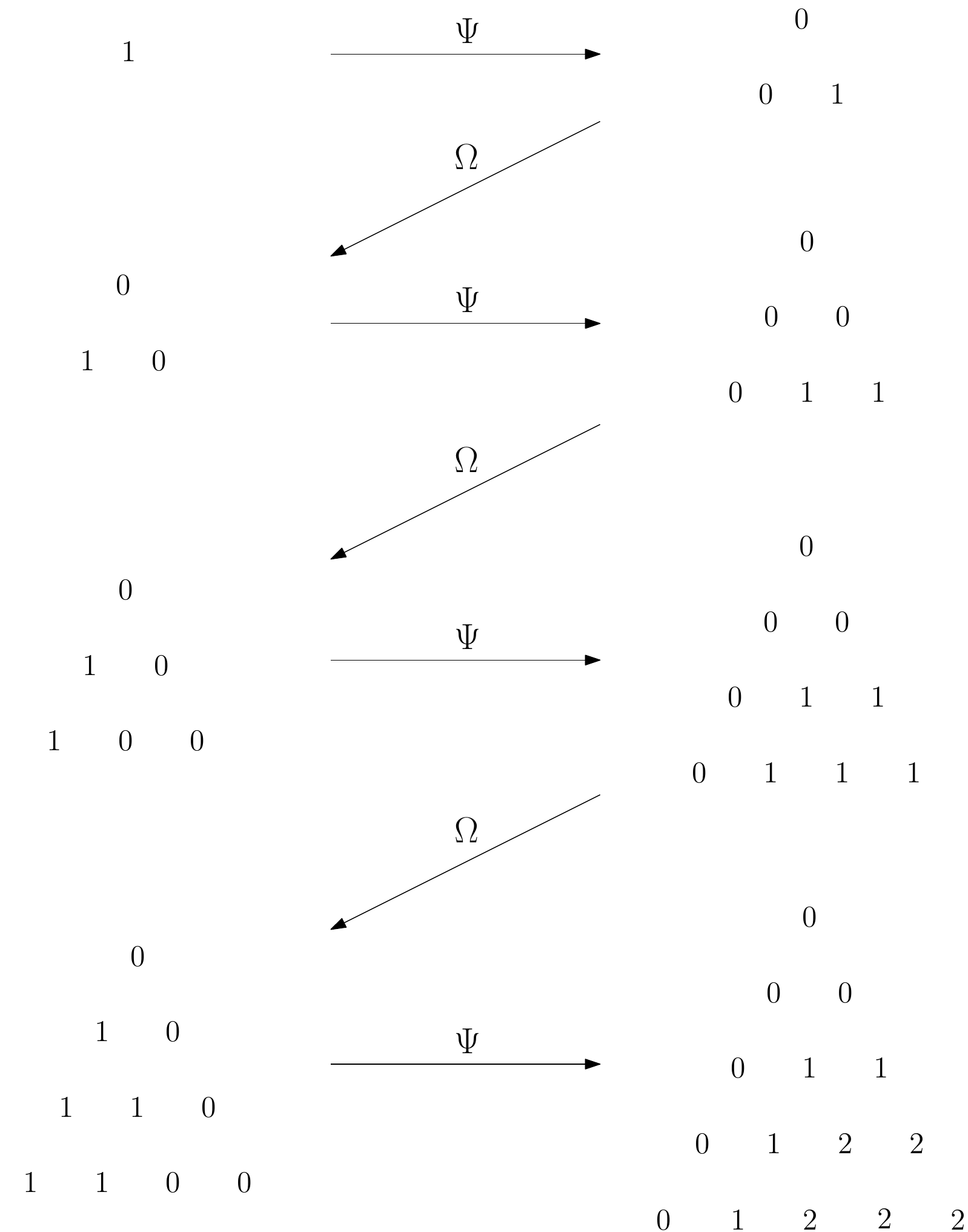}
\caption{The first five triangles of numbers of the case $k=3$ of the multidimensional boustrophedon. Here the operator $\Psi$ takes the partial sums of each row (starting with the empty partial sum) and adds a $0$ at the tip of the triangle, while the operator $\Omega$ rotates the triangle by $120$ degrees clockwise.}
\label{fig:boustrophedondim2}
\end{figure}

\begin{rem}
In the language of~\cite{R17}, the map $\Psi$ would be called $\Phi_{1,1,d+1}$. Note that the $\Phi$ operators of~\cite{R17} were acting on generating functions, while the $\Psi$ operator here acts on arrays of numbers. These two points of view are essentially the same, since there is a natural correspondence between arrays of numbers and their generating functions.
\end{rem}

We also introduce the operator $\Omega$, which sends an array of numbers $(b_{\underline{i}})_{\underline{i}\in T_N^d}$ to the array of numbers $(c_{\underline{i}})_{\underline{i}\in T_N^d}$, where for $\underline{i}\in T_N^d$, we have
\begin{equation}
c_{(i_1,\ldots ,i_{d+1})}:=b_{(i_{d+1},i_1,i_2,\ldots,i_d)}.
\end{equation}
The operator $\Omega$ acts by cyclically permuting the indices.

We can now state the following recurrence relation for the arrays of numbers $(a_{\underline{i}})_{\underline{i}\in T_{n+1}^{k-1}}$:

\begin{thm}
For $n\geq3$ and $2 \leq k \leq n-1$, we have
\begin{equation}
\label{eq:boustrophedon}
(a_{\underline{i}})_{\underline{i}\in T_{n+1}^{k-1}}= \Omega \circ \Psi (a_{\underline{i}})_{\underline{i}\in T_n^{k-1}}.
\end{equation}
\end{thm}

\begin{proof}

\begin{figure}[htbp]
\includegraphics[width=4in]{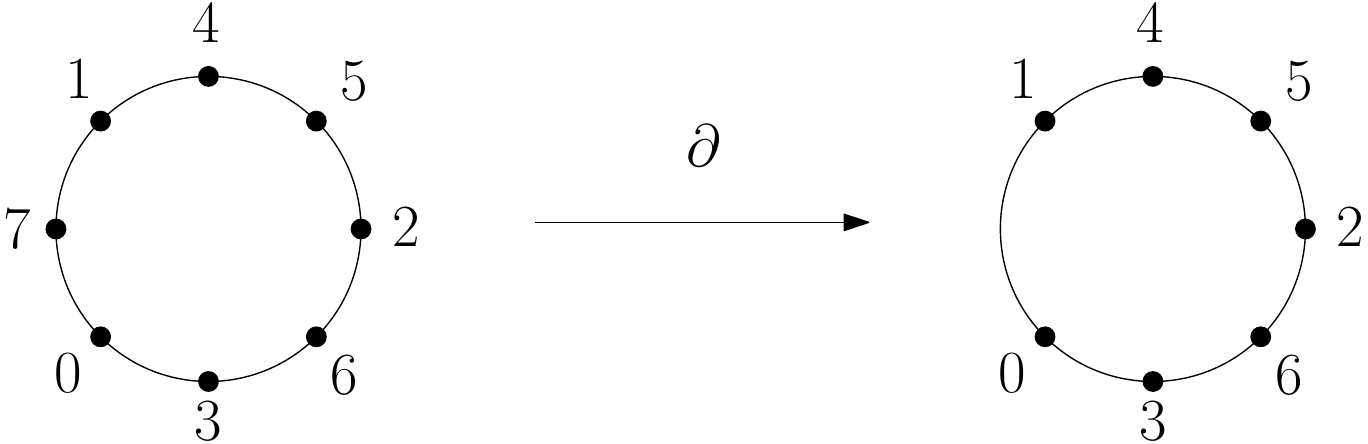}
\caption{Illustration of the action of $\partial$ on a total order on $[7]$, yielding a total order on $[6]$.}
\label{fig:deletion}
\end{figure}

Consider the map $\partial$ which to an element $\pi\in \check{A}_{\underline{i}}$ associates the element $\pi'\in \hat{A}_{k,n-1}$ obtained by deleting the number $n$ from the circle. For $\pi\in \check{A}_{\underline{i}}$ with $\underline{i}=(i_1,\ldots,i_k)$, the element $\partial(\pi)$ belongs to some $\check{A}_{\underline{j}}$, where $\underline{j}=(j_1,\ldots,j_k)$ satisfies the following conditions:
\begin{gather}
\label{eq:cond1} 1 \leq j_1 \leq i_k-1 \\
\label{eq:cond2} j_k=i_{k-1}+i_k-j_1-1 \\
\label{eq:cond3} j_m=i_{m-1} \text{ for } 2 \leq m \leq k-1. 
\end{gather}
Furthermore, the map $\partial$ is a bijection between $\check{A}_{\underline{i}}$ and $\sqcup_{\underline{j}} \check{A}_{\underline{j}}$, where the union is taken over all the multi-indices $\underline{j}$ satisfying conditions~\eqref{eq:cond1}-\eqref{eq:cond3}, because starting from any element $\pi'$ in some $\check{A}_{\underline{j}}$ with $\underline{j}$ satisfying conditions~\eqref{eq:cond1}-\eqref{eq:cond3}, there is a unique way to add back the number $n$ on the circle to obtain $\pi$ such that $L_\pi(n-1,n)=i_{k-1}$. This concludes the proof. 
\end{proof}

We can use this to compute the cardinality of any $\hat{A}_{k,n}$ inductively on $n$. We start at $n=k-1$. In this case, the simplex $T_k^{k-1}$ has a single element, and we start with the array consisting of a single entry equal to $1$. Then we apply formula~\eqref{eq:boustrophedon} to reach the desired value of $n$, and we take the sum of all the entries in the corresponding array of numbers.

\begin{rem}
In the case $k=2$, we recover the classical boustrophedon used to compute the Entringer numbers, which are the numbers $a_{\underline{i}}$ (see e.g.~\cite{R17}). The appearance of the operator $\Omega$ explains why each line is read alternatively from left to right or from right to left. See Figure~\ref{fig:boustrophedondim1} for the computation of the first four lines of the classical boustrophedon. For $k>2$, the numbers $a_{\underline{i}}$ may be seen as higher-dimensional versions of the Entringer numbers and the numbers
\begin{equation}
a_{k,n}:=\sum_{\underline{i}\in T_{n+1}^{k-1}} a_{\underline{i}}
\end{equation}
may be seen as higher-dimensional Euler numbers (where the number $k$ is the dimension parameter). See Figure~\ref{fig:boustrophedondim2} for the computation of the first five triangles of numbers in the case $k=3$.
\end{rem}

\section*{Acknowledgements}
MJV thanks Alejandro Morales for a fruitful discussion. AA was partially supported by the UGC Centre for Advanced Studies and by Department of Science and Technology grant EMR/2016/006624. 
SR was supported by the Fondation Simone et Cino del Duca. SR also acknowledges the hospitality of the Faculty of Mathematics of the Higher School of Economics in Moscow, where part of this work was done.

\bibliographystyle{alpha}
\bibliography{bibliographie}

\setlength{\parindent}{0mm}

\end{document}